\documentclass[]{article}
\usepackage{array}
\usepackage{graphicx}
\usepackage{graphics}
\usepackage{color,soul}
\usepackage{amssymb,amsmath,amsthm,epsfig}
\usepackage{latexsym}
\newcommand{\J}{P^{(\alpha,\alpha)}}
\newcommand{\wJ}{\widetilde P^{(\alpha,\alpha)}}
\usepackage{color}
\newtheorem{theorem}{Theorem}
\newtheorem{lemma}{Lemma}

\newtheorem{proposition}{Proposition}

\newtheorem{remark}{Remark}
\newcommand{\N}{\mathbb N}
\newcommand{\Z}{\mathbb Z}
\newcommand\ds{\displaystyle}
\newcounter{rea}
\newcounter{rek}
\newcommand{\ps}{\psi_{n,c}^{(\alpha)}}

\newcommand{\norm}[1]{{\left\|{#1}\right\|}}
\begin{document}
\begin{center}
{\large {\bf  Further Spectral Properties of the Weighted Finite Fourier Transform Operator and Approximation in Weighted Sobolev Spaces.}}\\
\vskip 1cm NourElHouda Bourguiba$^a$ and Ahmed Souabni$^a$ {\footnote{
Corresponding author: Ahmed Souabni, email: souabniahmed@yahoo.fr\\
This work was supported in part by the  Tunisian DGRST  research grants  UR 13ES47.}}
\end{center}

\vskip 0.25cm {\small
\noindent
$^a$ University of Carthage, Department of Mathematics, Faculty of Sciences of Bizerte, Tunisia.}\\

\noindent{\bf Abstract}---
In this work, we first give some mathematical preliminaries concerning the generalized prolate spheroidal wave function (GPSWFs). This set of special functions have been introduced in \cite{Wang2} and \cite{Karoui-Souabni1} and they are defined as the infinite and countable set of the  eigenfunctions of a weighted finite Fourier transform operator. Then, we show that the set of the singular values of this operator has a super-exponential decay rate. We also give some   local estimates and bounds  of these  GPSWFs. As an application of the spectral properties
of the GPSWFs and their associated eigenvalues, we give their quality of approximation in a weighted Sobolev space.Finally, we provide the reader with some numerical examples that illustrate the different results of this work.\\

\section{Introduction}

We first recall that for $c>0$, the classical prolate spheroidal wave functions (PSWFs) were first discovered and studied by D. Slepian and his co-authors, see \cite{Slepian1, Slepian2}. These PSWFs are defined as the solutions of the following energy maximization problem :
\begin{equation*}
\mbox{ Find } g=\arg\max_{f\in B_c}\frac{\int_{-1}^1|f(t)|^2 dt}{\int_{\mathbb{R}}|f(t)|^2 dt}
\end{equation*}
Where $B_c$ is the classical Paley-Winer space, defined by
\begin{equation}
\label{Bc}
B_c=\{ f\in L^2(\mathbb R),\,\, \mbox{Support } \widehat f\subseteq [-c,c]\}.
\end{equation}
Here, $\widehat f$ is the Fourier transform of $f\in L^2(\mathbb R),$ defined by ${\displaystyle \widehat f(\xi)=\lim_{A\rightarrow +\infty}\int_{[-A,A]} e^{-i x\xi} f(x)\, dx.}$
Moreover, it has been shown in \cite{Slepian1},  that the PSWFs are also the eigenfunctions of the integral operator ${\displaystyle \mathcal Q_c}$ defined on $L^2(-1,1)$ by $${\displaystyle \mathcal Q_c f(x)=\int_{-1}^1 \frac{\sin(c(x-y))}{\pi(x-y)} f(y)\, dy} $$ as well as the eigenfunctions of a  commuting differential operator $ \mathcal{L}_c$ with $ \mathcal{Q}_c$ and given by
$$ \mathcal L_c f(x)= (1-x^2)f''(x)-2x f'(x) - c^2 x^2 f(x) $$
In this paper, we are interested in a weighted family of PSWFs called the generalized prolate spheroidal wave functions (GPSWFs), recently given in \cite{Karoui-Souabni1} and\cite{Wang2} . They are defined as the eigenfunctions of the  Gegenbauer perturbed differential operator,
\begin{equation*}
\mathcal L^{(\alpha)}_c \varphi(x)= (1-x^2)\varphi''(x)-2\left(\alpha+1\right)x\varphi'(x) - c^2 x^2 \varphi(x)=
\mathcal L_0 \varphi(x)-c^2 x^2 \varphi(x), \mbox{ where } \alpha>-1; c>0,
\end{equation*}
as well as  the eigenfunctions of the following commuting  weighted finite Fourier transform integral operator,
\begin{equation} \label{integraloperator}
{\displaystyle \mathcal F_c^{(\alpha)} f(x)=\int_{-1}^1 e^{icxy}  f(y)\,(1-y^2)^{\alpha}\, dy,\, \alpha > -1.}
\end{equation}
The question of quality of approximation by PSWFs has attracted some interests. In fact the authors in \cite{chen} have given an estimate of the decay of the coefficients of the PSWFs series expansion of $f \in H^s(I)$. Here $ H^s(I)$ is the Sobolev space over $I=[-1,1]$ and of exponent $s>0$.  Later in \cite{Wang2} the authors studied the convergence of the expansion of functions $f \in H^s(I)$ in a basis of PSWFs. We should mention that the problem of the best choice of the value of the bandwidth $c>0$ arises here. Numerical answers were given in \cite{Wang2}. Recently in \cite{Bonami-Karoui4}, the authors have given a precise answer to the previous problem. In this work, we further study the convergence rate of the projection ${\displaystyle S_N \cdot f = \displaystyle \sum_{k=0}^N <f,\ps> \ps }$ to $f$ . Here $ \ps$ are the GPSWFs and $ f \in H^s_{\alpha}(I),$ a weighted Sobolev space defined by:
$$  H^r_{\alpha
}(I) = \{f \in L^2_{I,\omega} : \norm{f}_{\alpha, r}<\infty \} \quad
\mbox{where} \quad \norm{f}^2_{\alpha,r} = \displaystyle{\sum_{k=0}^r \norm{f^{(k)}}_{\omega}^2} $$
equipped with the norm and semi-norm
$$\norm{f}_{\alpha,r} = \displaystyle{\sum_{k=0}^r \norm{f^{(k)}}_{\omega}^2} \quad \mbox{and} \quad |f|_{\alpha,r} = \|f^{(k)}\|_{\omega_{\alpha}} $$
Here $$ \omega_\alpha(t) = (1-t^2)^{\alpha} \mbox{ and } \norm{.}_{\omega} = \norm{.}_{L^2(I,\omega_\alpha)} $$
In this work , we will give the following result

\begin{theorem}
	Let $c>1$ and $\alpha  >0$, then there exist constants $K>0$ and $a>0$ such that, when
	$N > m_{\alpha}c$ with $m_{\alpha}=4.13\left(1.28+\frac{2\alpha+1}{1.9}\right)^{0.55}$    and $f \in H^s_{\alpha}$; $s > 0$, we have the inequality
	\begin{eqnarray} \label{ineq}
	\|f-S_N(f)\|_{L^2(I, \omega_{\alpha})} &\leq& \left(1+ \left(\frac{N}{2}\right)^2\right)^{\frac{-s}{2}} \|f\|_{H^s_{\alpha}} +K.e^{-a N}\|f\|_{L^2(I,\omega_{\alpha})}\label{8}
	\end{eqnarray}
	Where $S_N(f)(t)=\ds{\sum_{n<N}}<f,\psi_{n,c}^{\alpha}>_{L^2_{\alpha}(I)}\psi_{n,c}^{\alpha}$
\end{theorem}
Note that the subject of spectral approximation in weighred Sobolev spaces has been studied  by Canuto and Quarteroni \cite{Canuto-Quarteroni} in the Legendre and Chebyshev case, Bernardi and Maday \cite{Bernardi-Maday}have developed Geganbauer approximations in the same context. Later in \cite{Guo-Wang} Guo and Wang developed the approximation result in the Jacobi case.
Our study of the convergence rate of $ \norm{S_N-Id}$ is done by using an estimate for the  decay of the eigenvalues $\lambda^{(\alpha)}_n(c)$ of the self-adjoint  operator ${\displaystyle \mathcal Q_c^{(\alpha)} =\frac{c}{2\pi}  \mathcal F_c^{(\alpha)^*} \mathcal F_c^{(\alpha)},}$ that is
$ \lambda_n^{(\alpha)}(c) = \frac{c}{2\pi}{|\mu_n^{(\alpha)}(c)|^2}$, where   $\mu_n^{(\alpha)}$ is the n-th eigenvalues of the integral operator given by \eqref{integraloperator}.  That's why the first part of this work will be devoted to the study of the spectral decay of the eigenvalues of the  operator ${\displaystyle \mathcal Q_c^{(\alpha)}}$. \\
Nonetheless, the decay rate of the eigenvalues of the integral operator has been studied in \cite{Karoui-Souabni2} where a first asymptotically decay rate of $ (\lambda_n^{(\alpha)})_n $ for $ 0<\alpha<3/2$ has been given . In this work we further improve this last result by giving clearly simpler and more general proof. We will prove that
\begin{theorem}
	For given real numbers $c > 0$,$ \;\; \;\alpha\geq 0$ and for any integer $n > \frac{ec+1}{2},$ we have
	\begin{equation}\label{Eq3.1}
	| \mu_n^{(\alpha)}(c)| \leq  \frac{k_{\alpha}}{c^{\frac{\alpha}{2}+1}\log\left(\frac{2n-1}{ec}\right)}
	\left(\frac{ec}{2n-1}\right)^{n+\frac{\alpha}{2}},\;\;\; k_{\alpha}=\left(\frac{2}{e}\right)^{1+\frac{\alpha}{2}}\pi(\Gamma(\alpha+1))^{1/2}.
	\end{equation}
	and
	\begin{equation}\label{newdecay}
	\lambda_n^{(\alpha)}(c) \leq \frac{K_\alpha}{c^{\alpha+1}\log^2(\frac{2n-1}{ec})}\Big(\frac{ec}{2n-1}\Big)^{2n+\alpha} \quad \mbox{with} \quad K_{\alpha} = \frac{\pi}{2}\big(2/e \big)^{\alpha+2} \Gamma(\alpha+1)\nonumber
	\end{equation}
\end{theorem}
This work is organized as follows. In section 2,  we give some mathematical  preliminaries concerning the GPSWFs. In section 3, we give some further estimates of the GPSWFs and some estimates of the eigenvalues of both integral and differential operators associated with GPSWFs.In section 4, we study,the quality  of approximation by the GPSWFs of functions from a weighted Sobolev space. Finally, in section 5  we give some numerical examples that illustrate the different results of this work.
\section{Mathematical Preliminaries}

In this section, we give some mathematical preliminaries concerning some properties as well as the computation of the GPSWFs. These mathematical preliminaries are used to describe the different results of this work.
The GPSWFs are defined as the eigenfunctions of
the weighted finite Fourier transform operator $\mathcal F_c^{(\alpha)},$ defined  by
\begin{equation}\label{Eq1.1}
\mathcal F_c^{(\alpha)} f(x)=\int_{-1}^1 e^{icxy}  f(y)\,\omega_{\alpha}(y)\, dy,\quad \omega_{\alpha}(y)=(1-y^2)^{\alpha},\quad \alpha >0
\end{equation}
They are also the eigenfunctions of the
operator ${\displaystyle \mathcal Q_c^{(\alpha)}=\frac{c}{2\pi}\mathcal F_c^{({\alpha})^*} \circ \mathcal F_c^{(\alpha)}},$   defined on $L^2{(I, \omega_{\alpha})}$ by
\begin{equation}\label{EEq0}
\mathcal Q_c^{(\alpha)} g (x) = \int_{-1}^1 \frac{c}{2 \pi}\mathcal K_{\alpha}(c(x-y)) g(y) \omega_{\alpha}(y) \, dy,
\quad \mathcal K_{\alpha}(x)=\sqrt{\pi} 2^{\alpha+1/2}\Gamma(\alpha+1) \frac{J_{\alpha+1/2}(x)}{x^{\alpha+1/2}}.
\end{equation}
Here $ J_{\alpha}(\cdot)$ is the Bessel function of the first kind and order $ \alpha>-1.$
Moreover, the eigenvalues $\mu_n^{(\alpha)}(c)$ and $\lambda_n^{(\alpha)}(c)$ of $\mathcal F_c^{(\alpha)}$ and
$\mathcal Q_c^{(\alpha)}$ are related to each others by the identity
${\displaystyle \lambda_n^{(\alpha)}(c)= \frac{c}{2\pi } |\mu_n^{(\alpha)}(c)|^2}.$
Note that the previous two integral operators commute with the following Gegenbauer-type Sturm-Liouville operator $\mathcal L_c^{(\alpha)},$ defined by
$$
\mathcal L_c^{(\alpha)} (f)(x)= -\frac{1}{\omega_{\alpha}(x)} \frac{d}{dx}\left[ \omega_{\alpha}(x) (1-x^2) f'(x)\right] +c^2 x^2  f(x),\quad \omega_{\alpha}(x)= (1-x^2)^{\alpha}.$$
Also, note that  the $(n+1)-$th eigenvalue $\chi_n^{\alpha}(c)$ of $\mathcal L_c^{(\alpha)}$ satisfies the following classical inequalities,
\begin{equation}
\label{boundschi}
n (n+2\alpha+1) \leq \chi_n^{\alpha}(c) \leq n (n+2\alpha+1) +c^2,\quad \forall n\geq 0.
\end{equation}
For more details, see \cite{Karoui-Souabni1}. We will denote by$(\ps)_{n\geq 0},$ the  set of the eigenfunctions of $\mathcal F_c^{(\alpha)}, \mathcal Q_c^{(\alpha)}$ and $\mathcal L_c^{(\alpha)}.$  They  are called generalized prolate spheroidal wave functions (GPSWFs).
It has been shown that $ \{ \ps , n\geq 0 \} $ is an orthogonal basis of $ L^2(I,\omega_{\alpha}), I=[-1,1].$
We recall that the restricted Paley-Wiener space of weighted $c-$band-limited functions has been defined in \cite{Karoui-Souabni1} by
\begin{equation}\label{GBc}
B_c^{(\alpha)}=\{ f\in L^2(\mathbb R),\,\,  \mbox{Support } \widehat f\subseteq [-c,c],\, \, \widehat f\in
L^2\big((-c,c), \omega_{- \alpha}(\frac{\cdot}{c})\big)\}.
\end{equation}
Here,  $L^2\big((-c,c), \omega_{- \alpha}(\frac{\cdot}{c})\big)$ is the weighted $L^2(-c,c)-$space with norm given by
$$\| f\|^2_{L^2\big((-c,c), \omega_{- \alpha}(\frac{\cdot}{c})\big)}= \int_{-c}^c |f(t)|^2 \omega_{-\alpha}\left(\frac{t}{c}\right)\, dt.$$
It has been shown in  \cite{Karoui-Souabni1} that the eigenvalues
$\lambda_n^{(\alpha)}(c)$ decay asymptotically at a super-exponential rate that is for $ 0<b<4/e$ there exists $N_b \in \mathbb{N}$ such that :
\begin{equation} \label{decaylambda0}
\lambda_n^{(\alpha)}(c) < e^{-2n \log(\frac{bn}{c})} \qquad \forall n \geq N_b,
\end{equation}
Note that a second decay rate of the  $ \lambda_n^{(\alpha)}(c)$ and valid for $0<\alpha <\frac{3}{2},$ has been recently given  in \cite{Karoui-Souabni2}. More precisely, it has been shown in \cite{Karoui-Souabni2}, that if  $c>0$ and $0<\alpha <\frac{3}{2}$, then there exist
$N_{\alpha}(c)\in \mathbb N$ and a constant $C_{\alpha}>0$ such that
\begin{equation}
\label{decay_lambda2}
\lambda_n^{(\alpha)}(c) \leq C_{\alpha} \exp\left(-(2n+1)\left[ \log\left(\frac{4n+4\alpha+2}{e c}\right)+ C_{\alpha} \frac{c^2}{2n+1}\right]\right),\quad
\forall\, n\geq N_{\alpha}(c).
\end{equation}
In \cite{Karoui-Souabni1}, the authors have proposed the following  scheme for the computation of the GPSWFs. In fact,   since  $\psi^{(\alpha)}_{n,c}\in L^2(I, \omega_{\alpha}),$ then its series expansion
with respect to the Jacobi polynomials basis is given by
\begin{equation}\label{expansion1}
\psi^{(\alpha)}_{n,c}(x) =\sum_{k\geq 0} \beta_k^n \wJ_k (x),\quad x\in [-1,1].
\end{equation}
Here, $\wJ_k$ denotes the normalized Jacobi polynomial of degree $k,$ given by
\begin{equation}\label{JacobiP}
\wJ_{k}(x)= \frac{1}{\sqrt{h_k}}\J_k(x),\quad h_k=\frac{2^{2\alpha+1}\Gamma^2(k+\alpha+1)}{k!(2k+2\alpha+1)\Gamma(k+2\alpha+1)}.
\end{equation}
The expansion coefficients $\beta_k^n$ as well as the eigenvalues $\chi_n^\alpha(c)$  this system is given by
\begin{eqnarray}\label{eigensystem}
\lefteqn{\frac{\sqrt{(k+1)(k+2)(k+2\alpha+1)(k+2\alpha+2)}}{(2k+2\alpha+3)\sqrt{(2k+2\alpha+5)(2k+2\alpha +1)}} c^2 \beta_{k+2}^n
	+ \big( k(k+2\alpha+1)+c^2 \frac{2k(k+2\alpha+1)+2\alpha-1}{(2k+2\alpha+3)(2k+2\alpha-1)} \big)
	\beta_k^n}\nonumber  \\
&&\hspace*{2cm} + \frac{\sqrt{k(k-1)(k+2\alpha)(k+2\alpha-1)}}{(2k+2\alpha-1)\sqrt{(2k+2\alpha+1)(2k+2\alpha-3)}} c^2
\beta_{k-2}^n= \chi_n^{\alpha}(c) \beta_k^n, \quad k\geq 0.
\end{eqnarray}
We denote, for any $ r \in \N $, $H^r_{\alpha}(I)$ the Jacobi-weighted Sobolev space by :
$$  H^r_{\alpha}(I) = \{f \in L^2(I,\omega_{\alpha}) : \norm{f}_{\alpha,r}<\infty \} \quad
\mbox{where} \quad \norm{f}_{\alpha,r} = \displaystyle{\sum_{k=0}^r \norm{f^{(k)}}_{\omega}^2} $$
In particular, $ L^2(I,\omega_{\alpha}) = H^0_{\alpha}(I) $ and $ \norm{.}_{\alpha,0} = \norm{.}_{\omega} $. \\
For any real r, we define the weighted Sobolev space $H^r_{\alpha}$ by interpolation  $ H^r_{\alpha} = \Big[H^{[r]}_{\alpha},H^{[r]+1}_{\alpha}\Big]$ as it is done in \cite{Bergh-Lofstrom} and \cite{Bernardi-Maday}
We denote by $ \mathcal{C}_{\infty}(I) $ the set of infinitely differentiable functions defined on $ I $.
Let $H_0^s(I) $ be the completion of $ \mathcal{C}_{\infty}(I) $ in $H^s(I)$. $H_0^s(I) $ is a Hilbert space under the inner product. For $s\geq 0$, we define $H_0^s(I) $by interpolation as in \cite{Bergh-Lofstrom} and \cite{Bernardi-Maday}. \\

\section{Further Estimates of the GPSWFs and their associated eigenvalues.}

In this section, we first give some new estimates of the eigenvalues $\mu_n^\alpha(c)$  associated with the
GPSWFs. Then, we study a local estimate of these GPSWFs.

\subsection{Estimates  of the eigenvalues.}

We first prove in a fairly simple way, the  super-exponential decay rate of the eigenvalues of the weighted Fourier transform operator.
We should mention that  the techniques used in this  proof is inspired from the recent paper \cite{Bonami-jamming-Karoui},
 where a similar result has been given in the special case $ \alpha = 0$.

\begin{theorem}\label{decaylambda}
	For given real numbers $c > 0$,$ \;\; \;\alpha\geq 0$ and for any integer $n > \frac{ec+1}{2},$ we have
	\begin{equation}\label{Eq3.1}
	| \mu_n^{(\alpha)}(c)| \leq  \frac{k_{\alpha}}{c^{\frac{\alpha}{2}+1}\log\left(\frac{2n-1}{ec}\right)}
	\left(\frac{ec}{2n-1}\right)^{n+\frac{\alpha}{2}},\;\;\; k_{\alpha}=\left(\frac{2}{e}\right)^{1+\frac{\alpha}{2}}\pi(\Gamma(\alpha+1))^{1/2}.
	\end{equation}
	and
	\begin{equation}\label{newdecay}
	\lambda_n^{(\alpha)}(c) \leq \frac{K_\alpha}{c^{\alpha+1}\log^2(\frac{2n-1}{ec})}\Big(\frac{ec}{2n-1}\Big)^{2n+\alpha} \quad \mbox{with} \quad K_{\alpha} = \frac{\pi}{2}\big(2/e \big)^{\alpha+2} \Gamma(\alpha+1)
	\end{equation}
\end{theorem}

\noindent
{\bf Proof:} We first recall the Courant-Fischer-Weyl Min-Max variational principle concerning the eigenvalues
of a  compact self-adjoint operator $T$ on a Hilbert space $\mathcal G,$ with positive eigenvalues arranged in the decreasing order $\lambda_0\geq \lambda_1\geq \cdots\geq \lambda_n\geq \cdots ,$ then we have
$$\lambda_n = \min_{f\in S_n}\,\,\,  \max_{f\in S_n^{\perp},\|f\|_{\mathcal G}=1} < Tf, f>_{\mathcal G},$$
where $S_n$ is a subspace of $\mathcal G$ of dimension $n.$ In our case, we have $T=\mathcal F_c^{\alpha^*}\mathcal F_c^{(\alpha)},$ $\mathcal G= L^2(I,\omega_{\alpha}).$ We consider the special case of  $$S_n=\mbox{Span}\left\{\wJ_0, \wJ_1,\ldots,\wJ_{n-1}\right\}$$ and
$$f=\ds\sum_{k\geq n} a_k \wJ_k \in S_n^{\perp},\qquad
\parallel f\parallel_{L^2_{(I,\omega_\alpha)}}=\ds\sum_{k\geq n}|a_k|^2=1.$$
Here the $ \wJ_{k} $ are the normalized Jacobi polynomials given by \eqref{JacobiP}.
From [\cite{NIST} page 456], we have
\begin{equation}\label{Eq3.2}
\| \mathcal F_c^{(\alpha)} \J_k\|_{L^2(I,\omega_{\alpha})} = \parallel \sqrt{\pi} \left(\frac{2}{cx}\right) ^{\alpha+1/2} \frac{\Gamma(k+\alpha+1)}{\Gamma(k+1)}J_{k+\alpha+1/2}(cx)
\parallel_{L^2(I,\omega_{\alpha})}.
\end{equation}
and
\begin{equation}\label{5}
\parallel  x^k
\parallel_{L^2(I,\omega_{\alpha})}^2=\beta(k+1/2,\alpha+1)\leq \sqrt{\frac{\pi}{e}}\frac{\Gamma(\alpha+1)}{k^{\alpha+1}}.
\end{equation}
By using the well-known  bound of the Bessel function given by
\begin{equation}\label{1}
|J_{\alpha}(x)|\leq \frac{|x|^{\alpha}}{2^{\alpha}\Gamma(\alpha+1)}\;\;\;\;\;\;\;\;\forall\alpha>-1/2, \;\; \forall x \in \mathbb{R},
\end{equation}
one gets,
\begin{eqnarray}
\| \mathcal F_c^{(\alpha)} \wJ_k\|_{L^2(I,\omega_{\alpha})}
&\leq & \frac{\sqrt{\pi}\Gamma(k+\alpha+1)}{\sqrt{h_k}\Gamma(k+1)\Gamma(k+\alpha+3/2)}\left(\frac{c}{2}\right) ^{k} \parallel  x^k
\parallel_{L^2(I,\omega_{\alpha})}. \nonumber\\
&\leq & \frac{\sqrt{\pi}\Gamma(k+\alpha+1)}{\sqrt{h_k}\Gamma(k+1)\Gamma(k+\alpha+3/2)}\left(\frac{c}{2}\right) ^{k}\left(\frac{\pi}{e}\right)^{1/4}\frac{\sqrt{\Gamma(\alpha+1)}}{k^{\frac{\alpha+1}{2}}}
\end{eqnarray}
Next, we use \eqref{JacobiP} and the following useful inequalities for the Gamma function, see  \cite{Batir}
\begin{equation}\label{3}
\sqrt{2e}\left(\frac{x+1/2}{e}\right)^{x+1/2}\leq\Gamma(x+1)\leq \sqrt{2\pi}\left(\frac{x+1/2}{e}\right)^{x+1/2}\;,\;\;\; x>0.
\end{equation}
Then, we have
\begin{eqnarray}
\| \mathcal F_c^{(\alpha)} \wJ_k\|_{L^2(I,\omega_{\alpha})} &\leq & \frac{\sqrt{\pi}\sqrt{(k+\alpha+\frac{1}{2})\Gamma(k+2\alpha+1)}}{2^{\alpha+\frac{1}{2}}\sqrt{\Gamma(k+1)}\Gamma(k+\alpha+\frac{3}{2})}\left(\frac{c}{2}\right) ^{k}\left(\frac{\pi}{e}\right)^{\frac{1}{4}}\frac{\sqrt{\Gamma(\alpha+1)}}{k^{\frac{\alpha+1}{2}}}\nonumber\\
&=&\frac{\pi \sqrt{\Gamma(\alpha+1)}\big(k+2\alpha+\frac{1}{2}\Big)^{\frac{k}{2}+\alpha+\frac{1}{4
}}}{\sqrt{ec}2^{\alpha}k^{\frac{\alpha+1}{2}}\left(k+\frac{1}{2}\right)^{\frac{k}{2}+\frac{1}{4}}\left(k+\alpha+1\right)^{k+\alpha+\frac{1}{2}}}
\left(\frac{ce}{2}\right) ^{k+\frac{1}{2
}}\nonumber\\
&=&\frac{\pi \sqrt{\Gamma(\alpha+1)}}{\sqrt{ec}2^{\alpha}k^{\frac{\alpha+1}{2}}}\left(1+\frac{\alpha-\frac{1}{2}}{k+\alpha+1}\right)^{\frac{k}{2}+\alpha+\frac{1}{4}}
\left(\frac{ce}{2k+1}\right) ^{k+\frac{1}{2}}.
\end{eqnarray}
Since the function $g_a(x)=\left(1+\frac{a}{x}\right)^{a+x}$ for
$x\geq1$ and $a>-1$ is decreasing on $[1,+\infty[$. Then, we have
\begin{eqnarray}
\| \mathcal F_c^{(\alpha)} \wJ_k\|_{L^2(I,\omega_{\alpha})} &\leq &\frac{\pi \sqrt{\Gamma(\alpha+1)}}{\sqrt{c}k^{\frac{\alpha+1}{2}}}
\left(\frac{ce}{2k+1}\right) ^{k+\frac{1}{2}}.\label{24}
\end{eqnarray}
Hence, For the previous $f\in S_n^{\perp},$  and by using H\"older's inequality, combined with  the Minkowski's inequality for an infinite sums, and taking into account that  $\|f\|_{L^2(I,\omega_{\alpha})}= 1,$ so that $|a_k|\leq 1,$ for $k\geq n,$  one gets:
\begin{eqnarray}
|< \mathcal F_c^{\alpha^*}\mathcal F_c^{(\alpha)} f,f>_{L^2(I,\omega_{\alpha})}|&=&|< \mathcal F_c^{(\alpha)} f,\mathcal F_c^{(\alpha)}f>_{L^2(I,\omega_{\alpha})}| \leq \sum_{k\geq n} |a_k|^2 \| \mathcal F_c^{(\alpha)} \wJ_k\|^2_{L^2(I,\omega_{\alpha})}\\
\label{Eq3.3}
&\leq&\left(\frac{2^\frac{1}{4}\pi \sqrt{\Gamma(\alpha+1)}}{\sqrt{ec}} \sum_{k\geq n} \left(\frac{ec}{2k+1}\right)^{k+1/2}\frac{1}{k^{\frac{\alpha+1}{2}}}\right)^2
\end{eqnarray}
The decay of the sequence appearing in the previous sum, allows us to compare this later with its integral counterpart, that is
\begin{equation}\label{Eq3.4}
\sum_{k\geq n} \left(\frac{ec}{2k+1}\right)^{k+1/2}\frac{1}{k^{\frac{\alpha+1}{2}}}
\leq \int_{n-1}^{\infty}\frac{ e^{-(x+1/2) \log(\frac{2x+1}{ec})}}{x^{\frac{\alpha+1}{2}}}\, dx\leq
\int_{n-1}^{\infty}\frac{ e^{-(x+1/2)\log(\frac{2n-1}{ec})}}{(n-1)^{\frac{\alpha+1}{2}}}\, dx.
\end{equation}
Hence, by using (\ref{Eq3.3}) and (\ref{Eq3.4}), one concludes that
\begin{eqnarray}
\max_{f\in S_n^{\perp},\, \|f\|_{L^2(I,\omega_{\alpha})=1}} < \mathcal F_c^{(\alpha)} f, \mathcal F_c^{(\alpha)} f>^{1/2}_{L^2(I,\omega_{\alpha})}&\leq &\frac{2^\frac{1}{4}\pi \sqrt{\Gamma(\alpha+1)}}{\sqrt{ec}} \frac{1}{\log(\frac{2n-1}{ec})} \left(\frac{ec}{2n-1}\right)^{n-1/2}\frac{1}{(n-1)^{\frac{\alpha+1}{2}}}\nonumber\\
&\leq&\pi \sqrt{\Gamma(\alpha+1)} \left(\frac{2}{ec}\right)^{\frac{\alpha+2}{2}} \frac{1}{\log(\frac{2n-1}{ec})} \left(\frac{ec}{2n-1}\right)^{n+\frac{\alpha}{2}}\label{Eq3.5}
\end{eqnarray}
To conclude the proof of the theorem, it suffices to use the Courant-Fischer-Weyl Min-Max variational principle.$\qquad \Box $

\begin{remark}
	The decay rate given by the last proposition improve the two results mentioned above \eqref{decaylambda0} and \eqref{decay_lambda2}. Indeed, we have a more precise result furthermore the proof given in \cite{Karoui-Souabni2} contains two serious drawbacks. It is only valid in the particular case $0 < \alpha < 3/2 $ and clearly more sophisticated than the previous one.
\end{remark}

\subsection{Local estimates of the GPSWFs}

In this paragraph, we give a precise  local estimate of the GPSWFs, which is valid for $0\leq \alpha \leq 1/4.$  Then, this local estimate will be used to provide us with a new lower bound for the eigenvalues $\chi_{n}^\alpha(c)$ of the differential operator $\mathcal L_c^\alpha,$ for $0\leq \alpha \leq 1/4. $  For this purpose, we first  recall that $\ps $ are the bounded solutions of the following ODE :
\begin{equation}\label{eq2_diff}
\omega_{\alpha}(x)\mathcal L_c^{(\alpha)} \psi(x)+w_{\alpha}(x)\chi_{n,\alpha}\psi(x)= ( w_{\alpha}(x)\psi'(x)(1-x^2))'+ w_{\alpha}(x)(\chi_{n}^{\alpha}-c^2 x^2)\psi(x)=0,\quad x\in [-1,1].
\end{equation}
As it is done in \cite{Bonami-Karoui2}, we consider the incomplete elliptic integral
\begin{equation}\label{Liouvile_transform1}
S(x)=\int_x^1 \sqrt{\frac{1-qt^2}{1-t^2}}\, dt,\qquad q=\frac{c^2}{\chi_n^\alpha(c)}<1.
\end{equation}
Then, we  write $\psi_{n,c}^{(\alpha)}$ into the form
\begin{equation}\label{Liouvile_transform2}
\psi(x)=\phi_{\alpha}(x)V(S(x)),\qquad \phi_{\alpha}(x)=(1-x^2)^{(-1-2\alpha)/4} (1-qx^2)^{-1/4}.
\end{equation}
By combining (\ref{eq2_diff}), (\ref{Liouvile_transform2}) and using straightforward computations,
it can be easily checked that $V(\cdot)$ satisfies the following second order differential equation
\begin{equation}
V''(s)+\left(\chi_{n,\alpha}+\theta_{\alpha}(s)\right)V(s)=0,\quad s\in [0, S(0)]
\end{equation}
with $$ \theta_{\alpha}(S(x))=(w_{\alpha}(x)(1-x^2)\phi_{\alpha}'(x))'\frac{1}{\phi_{\alpha}(x)w_{\alpha}(x)(1-qx^2)}.$$
Define $ Q_{\alpha}(x)=w_{\alpha}(x)^2(1-x^2)(1-qx^2),$ then we have
${\displaystyle  \frac{\phi_{\alpha}'(x)}{\phi_{\alpha}(x)}=-1/4 \frac{Q_{\alpha}'(x)}{Q_{\alpha}(x)}. }$
It follows that $\theta_{\alpha}(s)$ can be written as
\begin{equation}
\theta_{\alpha}(S(x))=\frac{1}{16(1-qx^2)}\Big[\Big(\frac{Q_{\alpha}'(x)}{Q_{\alpha}(x)}\Big)^2(1-x^2)-4\frac{d}{dx}\Big((1-x^2)\frac{Q_{\alpha}'(x)}{Q_{\alpha}(x)}\Big)-4(1-x^2)\frac{Q_{\alpha}'(x)}{Q_{\alpha}(x)}
\frac{w_{\alpha}'(x)}{w_{\alpha}(x)}\Big].
\end{equation}
Since $ Q_{\alpha}(x)=w_{\alpha}^2(x)Q_0(x),$ then we have
${\displaystyle  \frac{Q_{\alpha}'(x)}{Q_{\alpha}(x)}=2\frac{w_{\alpha}'(x)}{w_{\alpha}(x)}+\frac{Q_0'(x)}{Q_0(x)}} $ and ${\displaystyle \frac{w_{\alpha}'(x)}{w_{\alpha}(x)}=-\frac{2\alpha x}{1-x^2}.} $
Hence, we have
\begin{eqnarray}
\theta_{\alpha}(S(x)) &=&\theta_0(S(x))+\frac{-1}{4(1-qx^2)}\Big[ \big(\frac{w_{\alpha}'(x)}{w_{\alpha}(x)}\big)^2 (1-x^2)+2\frac{d}{dx}[(1-x^2)\frac{w_{\alpha}'(x)}{w_{\alpha}(x)}]  \Big] \nonumber \\
&=&\theta_0(S(x))+ \frac{1}{(1-x^2)(1-qx^2)}\Big( -\alpha^2 x^2 + \alpha (1-x^2) \Big) \nonumber \\
&=&\theta_0(S(x))+ \frac{\alpha(1+\alpha)}{(1-qx^2)}-\frac{\alpha^2}{(1-q x^2)(1-x^2)}.
\end{eqnarray}
The previous equality allows us to prove the following lemma.
\begin{lemma}
	For any $0\leq  \alpha\leq \frac{1}{4}$ and $0<q=c^2/\chi_n^{\alpha}(c) <\frac{3}{17}$, we have $\theta_{\alpha}(S(x))$ is increasing on $ [0.1]$.\\
\end{lemma}
\noindent
{\bf Proof:}
We use the notation $u=1-x^2$ and with straightforward computations,  we have
\begin{eqnarray*}
(\theta_{\alpha}(S(x)))'&=&(\theta_{0}(S(x)))'+\frac{2x}{(1-qx^2)^2(1-x^2)^2}\left[q\alpha(\alpha+1)(1-x^2)^2
-q\alpha^2(1-x^2)-\alpha^2(1-qx^2)\right]  \\
&=& \frac{2xH(u)}{4u^2(1-q+qu)^4}
\end{eqnarray*}
where
\begin{eqnarray*}
H(u) &=& G(u)+4(1-q+qu^2)^2[q(\alpha^2+\alpha)u^2-\alpha^2(1-q+2qu)]\\
&\geq&(1-q+qu^2)^2[G(u)+4[q(\alpha^2+\alpha)u^2-\alpha^2(1-q+2qu)]]
\end{eqnarray*}
with $G(u)\geq\frac{(1-q)^2}{4}(4-4q+u(17q-3))\geq0,\;\;\forall q>0,\;\;\alpha>0$
then
\begin{eqnarray*}
H(u) &\geq&(1-q+qu^2)^2\left[ \frac{(1-q)^2}{4}(4-4q+u(17q-3))+4[-q\frac{\alpha^3}{\alpha+1}-(1-q)\alpha^2]\right] \\
&\geq& (1-q+qu^2)^2\left[(1-q)^3-4\alpha^2(1-q\frac{1}{\alpha+1})+\frac{(1-q)^2}{4}(17q-3)u\right].
\end{eqnarray*}
Then  if $0\leq q\leq\frac{3}{17}$ and $\alpha\leq \frac{1}{4}$ we have
\begin{eqnarray*}
H(u) &\geq& (1-q+qu^2)^2  \left[(1-q)^3-4\alpha^2(1-q\frac{1}{\alpha+1})+\frac{(1-q)^2}{4}(17q-3)\right]\nonumber\\
&\geq &(1-q+qu^2)^2\left[\frac{(1-q)^2}{4}(1+13q)-4\alpha^2\right]\geq0
\end{eqnarray*}
$\qquad \Box $
As a consequence of the previous lemma, one gets.
\begin{lemma}
	For any $0\leq  \alpha\leq \frac{1}{4}$ and for any integer $n\in \mathbb N,$ with  $0<q=\frac{c^2}{\chi_n^\alpha(c)} <\frac{3}{17}$, we have
	\begin{equation}\label{estimation}
	 \ds{\sup_{x\in[0,1]}}\sqrt{(1-x^2)(1-qx^2)}\omega_{\alpha}(x)|\psi_{n,c}^{(\alpha)}(x)|^2\leq|\psi_{n,c}^{(\alpha)}(0)|^2+
	\frac{|{\psi_{n,c}^{(\alpha)}}'(0)|^2}{\chi_n^{\alpha}(c)}= A^2 \leq 2 \alpha + 1 \;
	\end{equation}
	
\end{lemma}

\begin{proof}
	For $0\leq  \alpha\leq \frac{1}{4}$ and $0<q<\frac{3}{17}$ ,we have $ \theta_{\alpha} \circ S $ is increasing. Since $S$ is decreasing then $ \theta $ is also decreasing. We define $K(s)= |V(s)|^2 + \frac{1}{\chi_n^{\alpha}(c)+\theta(s)}|V'(s)|^2 $
	By the fact that $K(s)$ and $ \frac{1}{\chi_n^{\alpha}(c)+\theta(s)} $ has the same monotonicity and  $\theta$ is decreasing , we conclude that $ K $ is increasing.
	Then $$ |V(S(x))|^2 \leq |K(S(x)) | \leq |K(S(0)) | $$
	We remark easily that $ V(S(0)) = \psi_{n,c}^{(\alpha)}(0) $ and $ V'(S(0)) ={ \psi_{n,c}^{(\alpha)}}'(0) $ then we conclude for the first inequality of \eqref{estimation}. The second inequality is given later on remark \ref{rq}
\end{proof}	
Next, we give some lower and upper bounds for the eigenvalues $\chi_n^\alpha(c),$ the $n+1$th eigenvalues of the differential operator. This is given by the following proposition.

\begin{proposition}
	For c and n such that $0\leq  \alpha\leq \frac{1}{4}$ and $0<q<\frac{3}{17}$,
	we have :
	\begin{equation} \label{chi_n}
	n(n+2\alpha+1)+C_\alpha c^2 \leq \chi_n^{\alpha}(c) \leq n(n+2\alpha+1)+c^2
	\end{equation}
	 Where $$ C_{\alpha} = 2(2\alpha+1)^2+1 - 2(2\alpha+1)\sqrt{1+(2\alpha+1)^2}.$$
\end{proposition}

\begin{proof}
	By differentiating with respect to $c$ the differential equation satisfied by GPSWFs , one gets
	$$ (1-x^2) \partial_c {{\psi_{n,c}^{(\alpha)}}'}'(x) - 2(\alpha+1)x \partial_c {\psi_{n,c}^{(\alpha)}}'(x)
	+ (\chi_n^{\alpha}(c)-c^2x^2) \partial_c \psi_{n,c}^{(\alpha)}(x) + \Big( \partial_c \chi_n^{\alpha}(c)-2cx^2 \Big) \psi_{n,c}^{(\alpha)}(x) = 0 $$
	Hence, we have
	$$ \Big(\mathcal{L}^{(\alpha)}_c + \chi_n^{\alpha}(c)Id\Big).\partial_c \psi_{n,c}^{(\alpha)} + \Big( \partial_c \chi_n^{\alpha}(c)-2cx^2 \Big) \psi_{n,c}^{(\alpha)}(x) = 0 $$
It's well known that $\mathcal{L}^{(\alpha)}_c $ is a self-adjoint operator and have $ \psi_{n,c}^{(\alpha)}$ as eigenfunctions . Hence, we have
	$$ < \Big(\mathcal{L}^{(\alpha)}_c + \chi_n^{\alpha}(c)Id\Big).\partial_c \psi_{n,c}^{(\alpha)}; \psi_{n,c}^{(\alpha)}>_{\omega_{\alpha}} = 0 $$
It follows that
	$$ \int_{-1}^{1} \Big( \partial_c \chi_n^{\alpha}(c)-2cx^2 \Big) {\psi_{n,c}^{(\alpha)}} ^{2}(x) dx = 0. $$
	From the fact that $ \norm{\psi_{n,c}^{(\alpha)}}_{L^2(I,\omega_\alpha)} = 1 $, one gets
	\begin{equation}\label{derivee}
	\partial_c \chi_n^{\alpha}(c) = 2c \int_{-1}^{1} x^2 {\psi_{n,c}^{(\alpha)}}^2(x) \omega_{\alpha}(x) dx
	\end{equation}
	As it is done in \cite{Bonami-Karoui1}, we denote by $ A= \Bigg[\psi_{n,c}^{(\alpha)}(0) + \frac{{\psi_{n,c}^{(\alpha)}}'(0)}{\chi_n^{\alpha}(c)} \Bigg]^{1/2} $ and consider the auxiliary function $$ K_n(t) = - (1-t^2)^{2\alpha+1} {\psi_{n,c}^{(\alpha)}}^2 - \frac{(1-t^2)^{2\alpha+2}}{\chi_n^{\alpha}(c)(1-qt^2)}({\psi_{n,c}^{(\alpha)}}')^2 $$
	Straightforward computations give us $$ K'_n(t) = 2(2\alpha+1)t(1-t^2)^{2\alpha} {\psi_{n,c}^{(\alpha)}}^2(t) - {H(t) {\psi_{n,c}^{(\alpha)}}'(t)}^2 $$ with $ H(t) \geq 0 $  for $t\in[0,1]$. Hence
	$$ K'_n(t)  \leq 2(2\alpha+1)t\omega_{\alpha}^2(t) {\psi_{n,c}^{(\alpha)}}^2(t) \leq 2(2\alpha+1)t\omega_{\alpha}(t) {\psi_{n,c}^{(\alpha)}}^2(t).  $$
	So one has the inequality
	\begin{equation}
	\label{Estimate0}
	K_n(1) - K_n(0) = A^2 \leq 2(2\alpha+1) \int_{0}^{1} t{|\psi_{n,c}^{(\alpha)}(t) |}^2 \omega_{\alpha}(t) dt \leq (2\alpha+1) \Bigg[\int_{-1}^{1}t^2 |\psi_{n,c}^{(\alpha)}(t) |^2 \omega_{\alpha}(t) dt \Bigg] ^{1/2}.
	\end{equation}
	That is
	\begin{equation}\label{11}
	A^2 \leq (2\alpha+1) B^{1/2}.
	\end{equation}
	Remark that \eqref{estimation} implies that
	\begin{equation} \label{12}
	1-B \leq 2 A^2
	\end{equation}
By combining \eqref{11} and \eqref{12} we conclude that $B^{1/2} $ is bounded below by the largest solution of the equation
	$ X^2 + 2(2\alpha+1) X -1 = 0 $
	
\end{proof}

\begin{remark} \label{rq}
By using \eqref{11} and since $B\leq1$, one concludes that the constant $A,$ given in \eqref{estimation} satisfies $A^2\leq 2\alpha+1.$
\end{remark}

\section{Approximation by the GPSWFs in Weighted Sobolev spaces}

In this section, we study the issue of the quality of spectral approximation of a function $f \in H^s_{\alpha}(I)$ by its truncated GPSWFs series expansion. \\
Note that a different spectral approximation result by the GPSWFs has been already given in \cite{Wang2}. It is important to mention here that the spectral approximation given in \cite{Wang2} is done in different approach. More precisely, by considering the weighted Sobolev space associated with the differential operator defined by
$$ \widetilde{H}^{r}_{\omega_{\alpha}}(I) = \{ f\in L^2(I,\omega_{\alpha}): \|f\|_{\widetilde{H}^{r}_{\omega_{\alpha}}(I)} = \|(\mathcal{L}^{(\alpha)}_c)^{r/2}.f\|^2 = \sum_{k=0}^{\infty}(\chi_n^{\alpha})^r |f_k|^2 <\infty \} $$
where $f_k$ are expansions coefficients of $f$ in the GPSWFs's basis. Then it has been shown that: \\
for any $f \in \widetilde{H}^{r}_{\omega_{\alpha}}(I) $ with $ r\geq 0 $
$$ \| S_{N}.f-f \|_{L^2(I,\omega_{\alpha})} \leq \Big( \chi^{(\alpha)}_{N+1} \Big)^{-r/2} \|f\|_{\widetilde{H}^{r}_{\omega_{\alpha}}(I)} $$
For more details, we refer the reader to \cite{Wang2}.\\
\begin{remark}
	We have the following norm's comparison given in \cite{Wang2} :
	$$ \|f\|_{\widetilde{H}^{r}_{\omega_{\alpha}}(I)} \leq
	C (1+c^2)^{r/2} \|f\|_{r,\omega} $$
	Note that for large value of $N$ , we have the same decay rate but when $N^2$ and $ 1+c^2$ are comparable we can notice clearly the importance of the following result.
	
\end{remark}
\begin{theorem}
	For $\alpha \geq 0 $ . For any real number $m \geq 0 $, $ f \in H^m_{\alpha} (I) $ and for any integer number $N \geq 1 $, we have
	\begin{equation} \label{approx1}
	\norm{f-S_Nf}^2_{L^2(I,\omega_\alpha)} \leq C \Bigg[ \sum_{n=N+1}^{\infty} \Big( \chi_{n}^{(\alpha)} \Big)^{-2m} \Bigg] \norm{f}_{m+2,\alpha}^2
	\end{equation}

\end{theorem}
\begin{proof}
We recall here the expression of the differential operator $ \mathcal{L}_c^{(\alpha)}$  :
$$
\mathcal L_c^{(\alpha)} (f)(x)= -\frac{1}{\omega_{\alpha}(x)} \frac{d}{dx}\left[ \omega_{\alpha}(x) (1-x^2) f'(x)\right] +c^2 x^2  f(x) = -(1-x^2)f''(x) + 2(\alpha+1)xf'(x) + c^2x^2f(x) .$$
For any positive integer $r$ and function $ f\in H^r_{\alpha,0}$ , the derivatives of $ \mathcal{L}_c^{(\alpha)}.f$ are given by :

\begin{eqnarray}
\frac{d^r}{dx^r}[\mathcal{L}_c^{(\alpha)}.f](x) &=& -(1-x^2) \frac{d^{r+2}f}{dx^{r+2}}(x) + 2(r+\alpha+1)x\frac{d^{r+1}f}{dx^{r+1}}(x) + \Big[ r(2\alpha + r +1)+c^2 x^2 \Big] \frac{d^{r}f}{dx^{r}}(x) \nonumber \\
&+& 2 c^2rx \frac{d^{r-1}f}{dx^{r-1}}(x) + c^2  r(r-1) \frac{d^{r-2}f}{dx^{r-2}}(x)
\end{eqnarray}
This expression yields directly that :
\begin{eqnarray}
\norm{\mathcal{L}_c^{(\alpha)}.f}^2_{\alpha,r} &=& \sum_{m=0}^r \norm{\frac{d^r}{dx^r}[\mathcal{L}_c^{(\alpha)}.f]}^2_{L^2(I,\omega_\alpha)} \nonumber \\
& \leq C & \norm{f}^2_{\alpha,r+2}
\end{eqnarray}
Therefore $\mathcal{L}_c^{(\alpha)}$ is a continuous operator from $ H^{r+2}_{\alpha,0}(I)$ into $ H^r_{\alpha,0}(I)$. \\
Then, by induction on $ r $, we obtain
\begin{equation}
\norm{\Big[\mathcal{L}_c^{(\alpha)}\Big]^r.f}_{\alpha,s} \leq C \norm{f}_{\alpha,r+2}
\end{equation}
For $ f \in H^r_{\alpha}(I) $, we write :
$$ f = \sum_{n=0}^{\infty} \hat{f_n}\psi^{(\alpha)}_{n,c} \quad \mbox{with} \quad
\hat{f_n} = \int_{-1}^1 f(x) \psi^{(\alpha)}_{n,c}(x) \omega_{\alpha}(x) $$
Since $\mathcal{L}_c^{(\alpha)}$ is self adjoint on $L^2(I,\omega_{\alpha}) $,
\begin{eqnarray}
\hat{f_n} &=& \frac{1}{\chi_{n,\alpha}}\int_{-1}^1 \mathcal{L}_c^{(\alpha)}.\ps(x) f(x) \omega_{\alpha}(x) dx =  \frac{1}{\chi_{n,\alpha}}\int_{-1}^1 \mathcal{L}_c^{(\alpha)}.f(x) \ps(x)  \omega_{\alpha}(x) dx \nonumber \\
&=& ... = \frac{1}{(\chi_{n,\alpha})^r}\int_{-1}^1 \Big[\mathcal{L}_c^{(\alpha)}\Big]^r.f(x) \ps(x)  \omega_{\alpha}(x) dx
\end{eqnarray}
Thus
\begin{eqnarray}
|\hat{f_n}|^2 &=& \frac{1}{\Big(\chi_{n}^{(\alpha)}\Big)^{2r}}|\Big<\Big[\mathcal{L}_c^{(\alpha)}\Big]^r.f,\ps\Big>_{L^2(I,\omega_{\alpha})}|^2 \leq \frac{1}{\Big(\chi_{n}^{(\alpha)}\Big)^{2r}} \norm{\Big[\mathcal{L}_c^{(\alpha)}\Big]^rf}_{L^2(I,\omega_{\alpha})} \nonumber \\
&\leq & \frac{C}{\Big(\chi_{n}^{(\alpha)}\Big)^{2r}} \norm{f}^2_{r+2,\alpha}
\end{eqnarray}
By writing $ \norm{f-S_Nf}^2_{L^2(I,\omega_\alpha)} = \displaystyle \sum_{n=N+1}^{\infty} |\hat{f_n}|^2 $ one gets \eqref{approx1}
Then, we obtain the result for $u$ regular, then for $ u \in H^{r+2}_{\alpha} (I)$ by density and continuity. Finally, for $s$ between two even integers, we conclude by interpolation.
	\end{proof}
In the following, we recall from \cite{Karoui-Souabni1} that for $c > 0$ a fixed positive real number, then, for all positive integers
$n,\; k$ such that $q=c^2/\chi^{\alpha}_{n}(c) <1$ and $ k(k+2\alpha+1)+C'_{\alpha}c^2\leq \chi_n^{\alpha}(c)$ , we have
\begin{eqnarray}
|\beta_k^n|&\leq& C_{\alpha}\left(\frac{2\sqrt{\chi_{n}^{\alpha}(c)}}{c} \right)^k |\mu_{n}^{\alpha}(c)|.\label{9}
\end{eqnarray}
With $C'_{\alpha}$ a constant depends only on $\alpha$, and $ {\displaystyle C_{\alpha}=\frac{2^{\alpha}(3/2)^{3/4}(3/2+2\alpha)^{3/4+\alpha}}{e^{2\alpha+3/2}}}$
\begin{lemma}\label{decaybeta}
	Let $c>1$ and $\alpha  >0$, then  for all positive integers $n,\; k$ such that, $k\leq n/1.9$ and
	$n\geq m_{\alpha}c$ with $m_{\alpha}=4.13\left(1.28+\frac{2\alpha+1}{1.9}\right)^{0.55}$
	we have
	\begin{eqnarray}
	|\beta_k^n| &\leq & C_{c,\alpha}e^{-\delta n}\label{10}
	\end{eqnarray}
	with $\delta$ and $C_{c,\alpha}$ positive  constants .
\end{lemma}
\begin{proof}
By $(\ref{9})$ ,$(\ref{Eq3.1})$  , the inequality   $\chi_{n}^{\alpha}(c)\leq n(n+2\alpha+1)+c^2$ and  the previous remark,
one concludes that for $c>0$ and $\alpha  >-1$, and  for all positive integers $n,\; k$ such that, $k\leq n/A$ and $n\geq Ac$ with $A\geq 1.9$, we have,
\begin{eqnarray}
|\beta_k^n|  &\leq &C_{\alpha}\left(\frac{2\sqrt{\chi_{n}^{\alpha}(c)}}{c} \right)^k |\mu_{n}^{\alpha}(c)|\nonumber\\
	&\leq &\frac{C_{\alpha}}{c^{\frac{({\alpha+2})}{2}}\log(\frac{2n-1}{ec})}
	\left(\frac{2}{c}\sqrt{\left({n(n+2\alpha+1)+c^2}\right)} \right)^{{k}}\left(\frac{ec}{2n-1}\right)^{n+\frac{\alpha}{2}} \nonumber\\
	&\leq &\frac{C_{\alpha}\left(\frac{ec}{2n-1}\right)^{\frac{\alpha}{2}}}{c^{\frac{({\alpha+2})}{2}}\log(\frac{2n-1}{ec})}
	\left(\frac{2}{c}{\sqrt{\left(n(n+2\alpha+1)+\frac{n^2}{A^2}\right)}} \right)^{\frac{n}{A}}\left(\frac{ec}{2n-1}\right)^{n}\nonumber\\
&\leq &\frac{C_{\alpha}\left(\frac{ec}{2n-1}\right)^{\frac{\alpha}{2}}}{c^{\frac{({\alpha+2})}{2}}\log(\frac{2n-1}{ec})}
	\left(\frac{2}{c}n\sqrt{\left(1+\frac{1}{A^2}+\frac{2\alpha}{n}\right)} \right)^{\frac{n}{A}}\left(\frac{(ec)^A}{(2n-1)^A}\right)^{\frac{n}{A}}\nonumber\\
&\leq &\frac{C_{\alpha}\left(\frac{ec}{2n-1}\right)^{\frac{\alpha}{2}}}{c^{\frac{({\alpha+2})}{2}}\log(\frac{2n-1}{ec})}
	\left(\frac{2}{c}n\sqrt{1+\frac{1}{A^2}+\frac{2\alpha+1}{n}}\frac{(ec)^A}{(2n-1)^A}
\right)^{\frac{n}{A}}\nonumber\\
&\leq &\frac{C_{\alpha}\left(\frac{ec}{2n-1}\right)^{\frac{\alpha}{2}}}{c^{\frac{({\alpha+2})}{2}}\log(\frac{2n-1}{ec})}
	\left(2e^Ac^{A-1}\left(\frac{n}{2n-1}\right)^A\sqrt{1+\frac{1}{A^2}+\frac{2\alpha+1}{n}}n^{1-A}
\right)^{\frac{n}{A}}.
	\end{eqnarray}
For the appropriate value of $A=1,9$. We have
\begin{eqnarray}
  |\beta_k^n| &\leq &\frac{C_{\alpha}\left(\frac{ec}{2n-1}\right)^{\frac{\alpha}{2}}}{c^{\frac{({\alpha+2})}{2}}\log(\frac{2n-1}{ec})}
	\left(\frac{1}{2}e^{\frac{A}{A-1}}\left(1+\frac{1}{A^2}+\frac{2\alpha+1}{A}\right)^{\frac{1}{2(A-1)}}\frac{c}{n}
\right)^{\frac{n(A-1)}{A}}\nonumber\\
&\leq &\frac{C_{\alpha}\left(\frac{ec}{2n-1}\right)^{\frac{\alpha}{2}}}{c^{\frac{({\alpha+2})}{2}}\log(\frac{2n-1}{ec})}
	\left(4.13\left(1.28+\frac{2\alpha+1}{1.9}\right)^{0.55}\frac{c}{n}
\right)^{0.47n}.
\end{eqnarray}
Since for all $n\geq m_{\alpha}c$, with $m_{\alpha}=4.13\left(1.28+\frac{2\alpha+1}{1.9}\right)^{0.55}$ there  are two positive  constant $\delta>0$ and
$C_{c,\alpha}\geq \frac{C_{\alpha}\left(\frac{ec}{2n-1}\right)^{\frac{\alpha}{2}}}{c^{\frac{({\alpha+2})}{2}}\log(\frac{2n-1}{ec})}$.
Then we have  \eqref{10}
\end{proof}
In \cite{Nicaise} the author has shown that on $H^s_{\alpha,0}$ we can also use the norm
$$ \norm{v}_{\alpha,s} = \sum_{n\in \N}\Big(1+(n(n+2\alpha+1))^{2s}\Big) |\hat{v}_{n,\alpha} |^2 \quad \mbox{where} \quad v = \sum_{n=0}^{\infty} \hat{v}_{n,\alpha}\widetilde{P}_n^{(\alpha,\alpha)} $$.

\begin{theorem}
		Let $c>0$ and $\alpha  >0 $, then there exist generic constants $K>0$ such that, when
		$N >  m_{\alpha}c$ with $m_{\alpha}=4.13\left(1.28+\frac{2\alpha+1}{1.9}\right)^{0.55}$    and $f \in H^s_{\alpha}$, we have the inequality
		\begin{eqnarray} \label{ineq}
		\|f-S_N(f)\|_{L^2_{\alpha}(I)} &\leq& K\left(1+ \left(\frac{N}{2}\right)^2\right)^{\frac{-s}{2}} \|f\|_{H^s_{\alpha}} +K.e^{-\delta N}\|f\|_{L^2(I,\omega_{\alpha})}
		\end{eqnarray}		
\end{theorem}

\begin{proof}
Let $ f \in H^s_{\alpha,0}$, then
$$ f = \sum_{k \in \N} \hat{v}_{n,\alpha}\widetilde{P}_n^{(\alpha,\alpha)} = \sum_{k<N_0} \hat{v}_{n,\alpha}\widetilde{P}_n^{(\alpha,\alpha)} + \sum_{k>N_0}\hat{v}_{n,\alpha}\widetilde{P}_n^{(\alpha,\alpha)} = g+h .$$
We have directly that
\begin{equation}
\norm{h}_{L^2(I,\omega_{\alpha})}^2  \leq C (1+(\frac{N}{2})^2)^{-s}\|f\|_{H^S_{\alpha}}
\end{equation}
Moreover, we can write the function $g$ under the form
$
g = \displaystyle \sum_{n \in \N} <g,\ps>_{L^2(I,\omega_{\alpha})} \ps
$

\begin{eqnarray}
\norm{g-S_N.g}_{L^2(I,\omega_{\alpha})}^2 &=& \sum_{n \geq N+1} |<g,\ps>_{L^2(I,\omega_{\alpha})}|^2 \nonumber \\
&\leq& \sum_{n\geq N+1} \sum_{k<N_0} |\hat{v}_k|^2 \Big|<P^{(\alpha,\alpha)}_k,\ps>_{L^2(I,\omega_{\alpha})}\Big|^2 \nonumber \\
&=& \sum_{n\geq N+1} \sum_{k<N_0} |\hat{v}_k|^2 |\beta_k^n|^2 \nonumber \\
& \leq & C e^{-2 \delta N} \norm{g}_{L^2(I,\omega_{\alpha})}
\end{eqnarray}
The last inequality use lemma \ref{decaybeta}.
Then we prove this result for functions $ f \in H^s_{\alpha,0} $ and we extend it into $ H^s_{\alpha} $ by density.
\end{proof}

\begin{remark}
	The previous proposition is considered as the generalization of a similar result given in \cite{Bonami-Karoui4}, in the special case $ \alpha = 0 $
\end{remark}
In \cite{Bonami-Karoui4}, authors studied the spectral approximation of a function $f \in H^s_{per}(I)$. Where $H^s_{per}$ is the subspace of $H^s_{0}(I)$ that extend into 2-periodic functions of the same regularity. Recall that in this space we define the norm :
$$ \norm{f}_{H^s_{per}}^2 = \sum_{k\in \Z} \Big( (1+k\pi^2) \Big)^s |b_k(f)|^2 $$
With $ \displaystyle b_k(f) = \frac{1}{\sqrt{2}} \int_{-1}^{1} f(x) e^{-i\pi k x} dx $ is the coefficient of the Fourier series expansion of $f$. We further prove in the following proposition that the family of GPSWFs are also adapted for the approximation of functions in $f \in H^s_{per}$.

\begin{lemma} \label{decaycoeff}
	
	for $c\geq1$, $\alpha>0$,  for all positif real $n,k$ such that $k\leq 0.14n$ and $n\geq m_{\alpha}c$ with $m_{\alpha}=4.13\left(1.28+\frac{2\alpha+1}{1.9}\right)^{0.55}$.
	There exist $C_{\alpha,c},\delta>0$ such that
	\begin{eqnarray}
	|\langle e^{ik\pi x}, \psi_{n,c}^{\alpha}\rangle _{L^2(I,\omega_{\alpha})}| &\leq& C_{\alpha,c} e^{-\delta n}
	\end{eqnarray}
	
\end{lemma}
\begin{proof}
For all $A\geq 1,9 $	we have
	\begin{eqnarray}
	|\langle e^{ik\pi x}, \psi_{n,c}^{\alpha}\rangle _{L^2(I,\omega_{\alpha\alpha})}| &\leq& \displaystyle{ \sum_{m\geq 0} } |\beta_m^n||\langle e^{ik\pi x}, \tilde{P}_{m}^{(\alpha,\alpha)}\rangle _{L^2(I,\omega_{\alpha\alpha})}|\nonumber\\
	&\leq& \displaystyle {\sum_{m= 0} ^{[n/A]}}|\beta_m^n||\langle e^{ik\pi x}, \tilde{P}_{m}^{(\alpha,\alpha)}\rangle _{L^2(I,\omega_{\alpha\alpha})}| \nonumber \\
	&+ &    \displaystyle{ \sum_{m\geq {[n/A]}+1}} |\beta_m^n||\langle e^{ik\pi x}, \tilde{P}_{m}^{(\alpha,\alpha)}\rangle _{L^2(I,\omega_{\alpha\alpha})}| =I_1 +I_2
	\end{eqnarray}

	For $I_2$ since,$|\beta_m^n|\leq1$ and by using $(\ref{Eq3.2}), \;  (\ref{1})$ and $\eqref{24}$ , one gets
	\begin{eqnarray}
	I_2&\leq& \ds{ \sum_{m\geq {[n/A]}+1}} |\langle e^{ik\pi x}, \tilde{P}_{m}^{(\alpha,\alpha)}\rangle _{L^2(I,\omega_{\alpha\alpha})}|
	\leq \ds{ \sum_{m\geq {[n/A]}+1}}\sqrt{\frac{\pi}{2m+1 }}\left(\frac{{k\pi e}}{2m +1} \right)^m\nonumber\\
	&\leq& \ds{ \sum_{m\geq {[n/A]}+1}}\sqrt{\frac{\pi}{2[\frac{n}{A}]+3 }}\left(\frac{{k\pi e}}{2[\frac{n}{A}]+3} \right)^m
	\leq K\sqrt{\frac{\pi}{2c+1 }}\left(\frac{{k\pi e}}{2\frac{n}{A}+1} \right)^{\frac{n}{A}}
	\end{eqnarray}
	Where $K$ is a positive constant. \\
	It is clear that for $k\leq \frac{1.2n}{e\pi}\simeq 0.14n$ there is $b>0 $ such as
	\begin{eqnarray}
	I_2 &\leq &K\sqrt{\frac{\pi}{2c+1 }}\left(\frac{{k\pi e}}{1.2n+1} \right)^{\frac{n}{1.7}}
	\leq K e^{-bn}
	\end{eqnarray}
	For $I_1 $, we have $|\langle e^{ik\pi x}, \tilde{P}_{m}^{(\alpha,\alpha)}\rangle _{L^2(I,\omega_{\alpha\alpha})}|\leq 1$, so we use (\ref{10}). For  $n> m_{\alpha}c$ with $m_{\alpha}=4.13\left(1.28+\frac{2\alpha+1}{1.9}\right)^{0.55}$ we have
	\begin{eqnarray}
	I_1 &\leq & \ds {\sum_{m= 0} ^{[n/A]}}|\beta_m^n|\leq K_{\alpha} e^{-an}
	\end{eqnarray}
	
\end{proof}

\begin{proposition}
	Let $c>1$ and $\alpha  >0 $, then there exist constants $K>0$ and $a>0$ such that, when
	$N > m_{\alpha}c$ with $m_{\alpha}=4.13\left(1.28+\frac{2\alpha+1}{1.9}\right)^{0.55}$   and $f \in H^s_{per}$; $s > 0$, we have the inequality
	\begin{eqnarray} \label{ineq}
	\|f-S_N(f)\|_{L^2_{\alpha}(I)} &\leq& \left(1+ \left(\frac{N}{2}\right)^2\right)^{\frac{-s}{2}} \|f\|_{H^s_{per}} +K.e^{-a N}\|f\|_{L^2_{\alpha}(I)}\label{8}
	\end{eqnarray}
	Where $S_N(f)(t)=\ds{\sum_{n<N}}<f,\psi_{n,c}^{\alpha}>_{L^2_{\alpha}(I)}\psi_{n,c}^{\alpha}$
	
\end{proposition}

\begin{proof}
	For $\alpha \geq 0 $,we have $ H^s_{\alpha}(I) \subset H^s_{0}(I)$. Let $ f \in H^s_{per}$, then
	$$ f = \sum_{k \in \Z} b_k(f) \phi_k = \sum_{|k|<N_0} b_k(f) \phi_k + \sum_{|k|>N_0} b_k(f) \phi_k = g+h
	\quad \mbox{where} \quad \phi_k(x) = e^{ikx} $$.
	\begin{equation}
	\norm{h}_{L^2(I,\omega_{\alpha})}^2 \leq \norm{h}_{L^2(I)}^2 \leq C (1+(N\pi)^2)^{-s}\|h\|_{H^s_{per}}
	\end{equation}
	By using the fact that $ g \in L^2(I, \omega_{\alpha})$ and lemma \ref{decaybeta}, one gets
	\begin{equation}\label{devg}
	g = \displaystyle \sum_{n \in \N} <g,\ps>_{L^2(I,\omega_{\alpha})} \ps
	\end{equation}
	
	\begin{eqnarray}
	\norm{g-S_N.g}_{L^2(I,\omega_{\alpha})}^2 &=& \sum_{n \geq N+1} |<g,\ps>_{L^2(I,\omega_{\alpha})}|^2 \nonumber \\
	&\leq& \sum_{n\geq N+1} \sum_{|k|<N_0} |b_k(g)|^2 <\phi_k,\ps>_{L^2(I,\omega_{\alpha})} \nonumber \\
	& \leq & C e^{-\delta N} \norm{g}_{L^2(I,\omega_{\alpha})}
	\end{eqnarray}
	In the last inequality we use lemma \ref{decaycoeff}
\end{proof}

\section{Numerical results}

In this section,   we give
some numerical examples that illustrate the different results of this work.\\

\noindent

{\bf Example 1:}

In this example we illustrate the decay rate of the eigenvalues $(\lambda_{n}^{\alpha}(c))$ given by theorem \ref{decaylambda} for different values of $ \alpha$. In figure \ref{lambda1} (a), we have plotted the graph of $(\lambda_{n}^{\alpha}(c))$ and in figure \ref{lambda1} (b) we have plotted the graphs of $log(\lambda_{n}^{\alpha}(c))$ versus the graphs of $ -(2n+1) \log\Big(\frac{4n+4\alpha+2}{ec}\Big)$

\begin{figure}[h]\label{lambda1}
	\centering
	{\includegraphics[width=14cm,height=3.5cm]{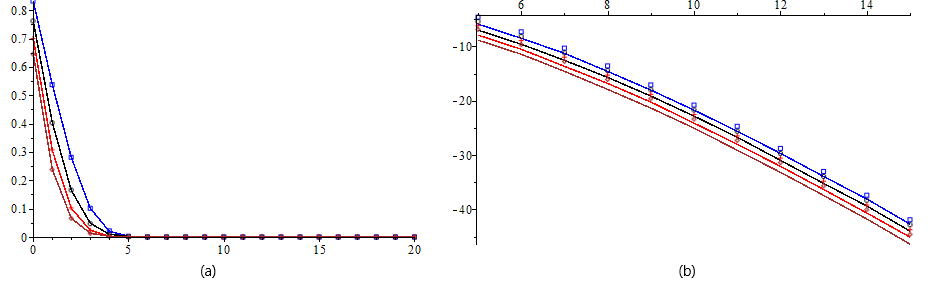}}
	\caption{(a) Graph of $\lambda_n^{(\alpha)}(c) $for $\alpha = 1$ (blue), $\alpha = 1.5$ (black), $\alpha = 2$ (red) and $\alpha = 2.5$ (brown) (b)Graph of $\log(\lambda_n^{(\alpha)}(c)) $} 	
\end{figure}

\noindent

{\bf Example 2:} In this example, for a positive real number $s$, we consider the Brownian motion function  
$$ B_s(x) = \sum_{k=1}^{\infty} \frac{X_k}{k^s} \cos(k \pi x) \quad x\in I=[-1,1] $$
Here $X_k$ is a sequence of independent standard Gaussian random variable. It is well known that $B_s\in H^{s'-1/2}(I),$ for any
$s'< s-\frac{1}{2}.$ Here, $H^s(I)$ is the  Sobolev space over $I,$ with smoothness exponent $s.$ For the special values of 
$s=1.5$ and  $ c= 5\pi $, and the two values of $N=46$ and $N=90,$ we have computed $S_N(B_s),$ the truncated series expansion to the order $N,$ of $B_s$ in the GPSWFs basis. We have found that 
$$\| B_s - S_{46}(B_s)\|_{\infty} \approx 01 E-01,\qquad \| B_s - S_{90}(B_s)\|_{\infty} \approx 05 E-02.$$
Moreover, the graphs of $B_{s}$ and  of its approximation  $S_{90} B_{s}$ are given by the figure \ref{brownian}.
\begin{figure}[h]\label{brownian}
	\centering
	{\includegraphics[width=18.05cm,height=6.5cm]{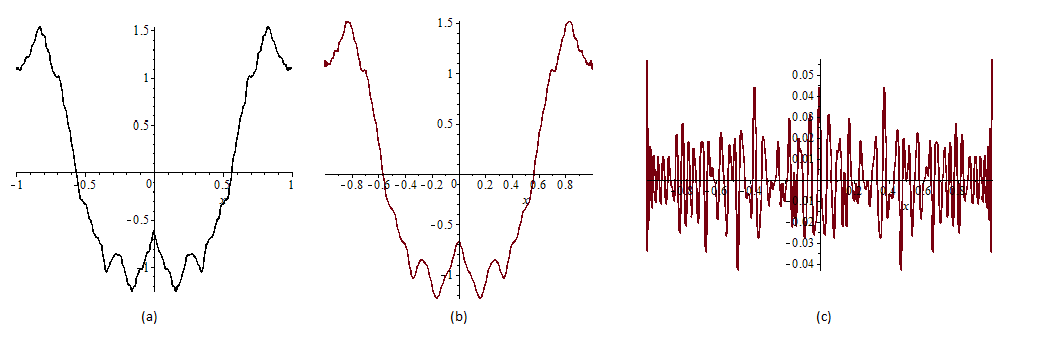}}
	\caption{(a) Graph of $B_s(x) $ (b) graph of $S^{(\alpha)}_{90} B_s(x) $  (c) graph of the error  $E_N= B_s(x)-S^{(\alpha)}_{90} B_s,$ with $s=1.5$,  $N=90.$, $c=5\pi$ and $\alpha=1.5$ }	
\end{figure}

\noindent

{\bf Example 2:}To illustrate approximation by GPSWFs we have considered the Weierstrass-Mondelbrot function given by :

$$ M_{s,\lambda}(x) = \sum_{k=0}^{\infty} \frac{\sin(\lambda^k x)}{\lambda^{k(2-s)}} $$
For $\lambda > 1 $ and $s<1$, $M_{s,\lambda} \in H^s_{\alpha}(I)$
In this example we have considered the values of $\alpha = 0.5$, $c=5\pi$ and $N=95$ then we have computed the projection, over $Span(\psi_{0,c}^{(\alpha)}, \cdots , \psi_{N,c}^{(\alpha)})$, $S_N(M_{s,\lambda}(x)) $. For this purpose, we have computed the different expansion coefficients $ \displaystyle{ C_n(M_{s,\lambda}(x)) = \int_{-1}^{1}M_{s,\lambda}(x)\ps(y) \omega_{\alpha}(y)dy  }$ which are exactly given thanks to \eqref{Eq3.2}. Since $M_{s,\lambda}$ is an odd function , $C_{2n}(M_{s,\lambda}) = $ . Moreover
\begin{eqnarray}
C_{2n+1}(M_{s,\lambda}) &=& \sum_{\ell=0}^{\infty}\beta_{2\ell+1}^{2m+1} \sum_{k=0}^{\infty} \frac{1}{\lambda^{k(2-s)}} \int_{-1}^{1} \sin(\lambda^k x) \hat{P}^{(\alpha)}_{2\ell+1}(x)\omega_{\alpha}(x) dx \nonumber \\
&=& \sqrt{\pi} 2^{\alpha+\frac{1}{2}} \sum_{\ell=1}^{\infty} (-1)^{\ell} \beta_{2\ell+1}^{2m+1} \frac{\Gamma(2\ell+\alpha+3)}{\sqrt(h_{2\ell+1})\Gamma(2\ell+1)} \sum_{k=1}^{\infty} \lambda^{k(s-\alpha-5/2)} J_{2\ell + \alpha+\frac{3}{2}}(\lambda^k)
\end{eqnarray}
In figure \ref{WM} we have plotted the function $M_{1,2}$ and the graph of the error $ M_{1,2} - S_N(M_{1,2}(x)) $
\begin{figure}[h]\label{WM}
	\centering
	{\includegraphics[width=18.05cm,height=6.5cm]{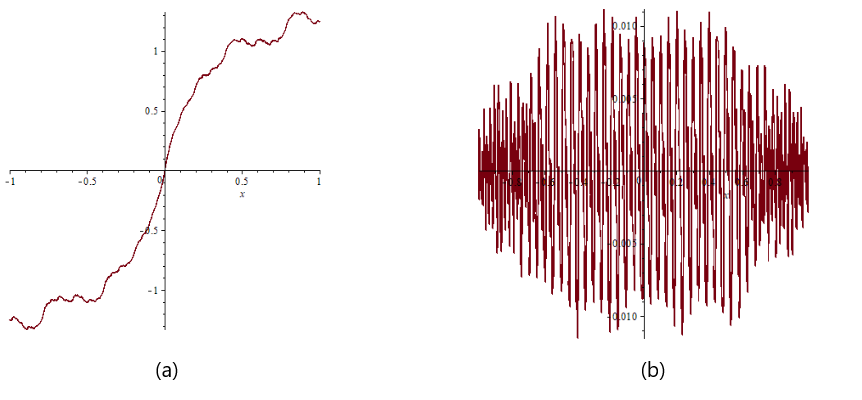}}
	\caption{(a) Graph of $M_{1,2}$ (b)  graph of the error $ M_{1,2} - S_N(M_{1.2})$} 	
\end{figure}

Moreover in table \ref{table}, we give the approximation error $\norm{ M_{s,2}-S_N^{(\alpha)}(M_{s,2})_{L^2(I,\omega_{\alpha})} } $ of the Weierstrass-Mandelbrot function by GPSWFs for different values of s and $\alpha$.

\begin{table}
	
	\begin{center}

		\begin{tabular}[h!]{c c c c c}
			\hline
			$\alpha$ & $s=0.25 $ & $s=0.5$ & $s=0.75$ & $s=1 $  \\
			\hline
			$0.1$ & $1.69146E-04 $ & $5.42800E-04$ & $1.74173E-03$ & $ 5.61554E-0.3 $  \\
			$0.5$ & $1.90589E-04 $ & $6.07253E-04$ & $1.93120E-03$ & $ 6.15556E-0.3 $  \\
			$1$ & $2.12572E-04 $ & $6.72661E-04$ & $2.12113E-03$ & $ 6.68912E-0.3 $ \\
			$1.5$ & $2.30518E-04 $ & $7.25472E-04$ & $2.27216E-03$ & $7.10411E-03$  \\
			$2.0$ & $2.45810E-04$ & $7.70063E-04$ & $2.39797E-03$ & $7.44278E-03$  \\
			\hline
			
		\end{tabular}
	\end{center}
	\caption{Values of the approximation of Weierstrass-Mandelbrot function $ \norm{ M_{s,2}-S_N^{(\alpha)}(M_{s,2})_{L^2(I,\omega_{\alpha})} } $ for different values of $s$ and $\alpha$}
	\label{table}
	
\end{table}


\begin{thebibliography}{9999}
\bibitem{Andrews} G. E. Andrews, R. Askey and  R. Roy, Special Functions,
Cambridge University Press , Cambridge, New York, 1999.

\bibitem{Bergh-Lofstrom} J.Bergh,J.Lofstrom, Interpolation spaces.An introduction,
Springer-Verlag , Cambridge, Berlin and New York, 1967.

\bibitem{Bernardi-Maday}C.Bernardi, Y.Maday, Spectral methods, in P.G,J.L.Lions(Eds.),Handbook of Numerical Analysis ,Vol5, Techniques of Scientific computing, Elsevier, Amesterdam,1997,pp.209-486.


\bibitem{Bonami-Karoui1} A. Bonami and A. Karoui,  Uniform bounds of prolate spheroidal wave functions
and eigenvalues decay, {\it  C. R. Math. Acad. Sci. Paris. Ser. I,} {\bf 352}  (2014), 229--234.

\bibitem{Bonami-Karoui2} A. Bonami and A. Karoui, Uniform approximation and explicit estimates of the Prolate
Spheroidal Wave Functions, {\it Constr. Approx.,} {\bf 43} (2016), 15--45.

\bibitem{Bonami-Karoui3} A. Bonami and A. Karoui, Spectral decay of time and frequency limiting operator, {\it Appl. Comput . Harmon. Anal.,} {\bf 42} (2017), 1--20.

\bibitem{Bonami-Karoui4} A. Bonami and A. Karoui, Approximation in Sobolev spaces by Prolate Spheroidal Wave Functions, {\it Appl. Comput . Harmon. Anal 42, 361-377} (2017)


\bibitem{Bonami-jamming-Karoui} A. Bonami, P.Jaming and A. Karoui, Non asymptotic behavior of the spectrum of the sinc kernel operator and related applications, available at  {\it arXiv}:1804.01257, (2018).

\bibitem{Canuto-Quarteroni} C.Canuto,A.Quarteroni, Approximation results for orthogonal polynomials in Sobolev spaaces,Math.Comp.38(1982)67-86.

\bibitem{chen} Q.Chen, D.Gottlieb,J.S Hesthaven, Spectral methods based on prolate spheroidal wave functions for hyperbolic PDEs, {\it SIAM J Numer. Anal} {\bf 43(5) } (2005) 1912-1933

\bibitem{Nicaise} Serge Nicaise, Jacobi Polynomials,weighted Sobolev spaces and approximation results of some singularities , Math.Nachr.213(2000),117-140


\bibitem{Guo-Wang} B.Y.Guo, L.L.Wang, Jacobi approximations in non-uniformly Jacobi-weighted Sobolev spaces, {\it J.Approx.Theory} {\bf 128 } (2004) 1-41



\bibitem{Karoui-Souabni1} A. Karoui and A. Souabni, Generalized Prolate Spheroidal Wave Functions: Spectral Analysis  and Approximation of Almost Band-limited Functions, {\it J.  Four.  Anal. Appl.,} {\bf 22} (2), (2016), 383--412.

\bibitem{Karoui-Souabni2}
 A. Karoui and  A. Souabni
\newblock{ Weighted Finite Fourier Transform Operator: Uniform
Approximations of the Eigenfunctions, Eigenvalues
Decay and Behaviour.}
J. Sci. Comp. {\bf 71} (2017), 547--570.



\bibitem{Landau} H. J. Landau and H. O. Pollak, Prolate spheroidal  wave functions, Fourier analysis and
uncertainty-III. The dimension of space of essentially time-and
band-limited signals, {\it Bell System Tech. J.} {\bf 41}, (1962),
1295--1336.

\bibitem{Moumni} T. Moumni, On essentially time and Hankel band-limited function {\it Integral transforms special functions,} {\bf 23} (2), (2012), 83--95.


\bibitem{NIST} Frank W. Olver, Daniel W. Lozier, Ronald F. Boisvert, and Charles W.
Clark, {\it NIST Handbook of Mathematical Functions,}  Cambridge University
Press, New York, NY, USA, 1st edition, 2010.



\bibitem{Slepian1} D. Slepian and H. O. Pollak, Prolate spheroidal wave functions, Fourier analysis and
uncertainty I, {\it Bell System Tech. J.} {\bf 40} (1961), 43--64.

\bibitem{Slepian2} D. Slepian,  Prolate spheroidal wave functions, Fourier analysis and
uncertainty--IV: Extensions to many dimensions; generalized
prolate spheroidal functions, {\it Bell System Tech. J.} {\bf 43}
(1964), 3009--3057.

\bibitem{Szego} G.  Szeg\"o,  {\it  Orthogonal polynomials,} Fourth edition, American Mathematical Society,
Colloquium Publications, Vol. XXIII. American Mathematical Society, Providence, R.I., 1975.


 \bibitem{Wang2} L. L. Wang and J. Zhang, A new generalization of the PSWFs with applications to spectral
 approximations on quasi-uniform grids, {\it Appl. Comput. Harmon. Anal.}
 {\bf 29},  (2010), 303--329.


\bibitem {Batir} Batir, N. {\it Inequalities for the gamma function.} Arch. Math. 91, 554--563 (2008).


\end{thebibliography}
\end{document}